\documentclass[12pt,a4paper]{article}
\usepackage{subfigure}
\usepackage{epsf,epsfig,amsfonts,amsgen,amsmath,amstext,amsbsy,amsopn,amsthm}
\usepackage{ifthen}     % For conditional statements
\usepackage{calc}       % For numerical calculations
\usepackage{indentfirst}
\usepackage{graphicx}
\usepackage{graphics,multicol}
\usepackage{color}
\usepackage[colorlinks=true,anchorcolor=blue,filecolor=blue,linkcolor=blue,urlcolor=blue,citecolor=blue]{hyperref}
\usepackage{floatrow}
\usepackage{enumerate}
\usepackage{tikz}
\usepgflibrary{bbox}
\usepackage{enumitem}
\usetikzlibrary{calc}
\usepackage{caption}
\usepackage{geometry}
\newtheorem{thm}{Theorem}[section]

\newtheorem{lem}{Lemma}[section]

\newtheorem{conj}{Conjecture}[section]

\newtheorem{claim}{Claim}

\newtheorem{case}{Case}

\theoremstyle{definition}% normal

\newenvironment{proofoftheorem}[1][Proof of Theorem]{\noindent\textbf{#1.}\ }{\qed}
\makeatletter
\@addtoreset{case}{lem}
\makeatother

\makeatletter
\@addtoreset{claim}{section}
\makeatother

\allowdisplaybreaks\allowdisplaybreaks[4]

\addtocounter{section}{0}

\makeatletter
\renewcommand\subsection{\@startsection{subsection}{2}{\z@}%
                                      {-5ex \@plus -1ex \@minus -.2ex}%  % Before the title
                                      {2.5ex \@plus .2ex}%              % After the title
                                      {\normalfont\itshape}}
\makeatother

\usetikzlibrary{backgrounds}
\usetikzlibrary{arrows}
\pgfdeclarelayer{edgelayer}
\pgfdeclarelayer{nodelayer}
\pgfsetlayers{background,edgelayer,nodelayer,main}
\tikzstyle{none}=[inner sep=0mm]
\tikzstyle{bluenode}=[fill=blue, draw=black, shape=circle, minimum
size=0.2cm,
inner sep=0pt]
\tikzstyle{whitenode}=[fill={rgb,255: red,245; green,245;
blue,245}, draw=black, shape=circle, minimum size=0.5cm, inner
sep=1pt]
\tikzstyle{yellownode}=[fill=yellow, draw=black, shape=circle, minimum size=0cm, inner sep=1pt]
\tikzstyle{pinknode}=[fill={rgb,255: red,255; green,191; blue,191}, draw=black, shape=circle, minimum size=0cm, inner sep=1pt]
\tikzstyle{blacknode}=[fill=black, draw=black, shape=circle, minimum size=0.2cm, inner sep=0pt]
\tikzstyle{rednode}=[fill={rgb,255: red,244; green,0; blue,0}, draw=black,
shape=circle, minimum size=0.2cm, inner sep=0.1pt]
\tikzstyle{square}=[draw=black, shape=rectangle, minimum
size=0.1cm, inner sep=3pt,fill={rgb,255: red,245; green,245;
blue,245},scale=0.5]
\tikzstyle{dot}=[fill=black, draw=black, shape=circle, minimum size=0.04cm, inner sep=0pt]
\tikzstyle{blackedge}=[line width=1.2pt, black]
\tikzstyle{blackedge_thick}=[-, draw=black, line width=1.5pt,
fill=none]
\tikzstyle{rededge}=[-, draw=red]
\tikzstyle{rededge_thick}=[-, line width=0.45mm, draw=red]
\tikzstyle{blackedge_opacity}=[-, -, draw={rgb,255: red,91; green,87; blue,84},
fill=none, line width=0.4mm, opacity=0.7]
\tikzstyle{balck_dash}=[-, dash pattern=on 0.2mm off 0.2mm]
\tikzstyle{blue_thick}=[-, line width=0.5mm, draw=blue]
\tikzstyle{blueedge}=[-, line width=1pt, blue, opacity=0.7]

\newcommand{\gridgraph}[2]{
    \begin{tikzpicture}[scale=2, every node/.style={circle, draw, minimum size=6pt, inner sep=2pt}]
        \foreach \i in {0,...,#1} {  % Rows (P_m)
            \foreach \j in {0,...,#2} {  % Columns (P_n)
                \node (v\i\j) at (\j,-\i) {};
            }
        }
        \foreach \i in {0,...,#1} {
            \foreach \j in {0,...,\numexpr#2-1\relax} {
                \draw [thick] (v\i\j) -- (v\i\the\numexpr\j+1\relax);
            }
        }
        \foreach \i in {0,...,\numexpr#1-1\relax} {
            \foreach \j in {0,...,#2} {
                \draw [thick] (v\i\j) -- (v\the\numexpr\i+1\relax\j);
            }
        }
    \end{tikzpicture}
}

\begin{document}
\title{The minimum size of 2-connected chordal bipartite graphs\footnote{Supported by the National Natural Science Foundation of China   (Grant No. 12271157, 12371346) and
the Natural Science Foundation of Hunan Province, China (Grant No. 2022JJ30028, 2023JJ30178).}}
\author{ 
{\bf Licheng Zhang$^{a}$} \thanks{lczhangmath@163.com}, {\bf Yuanqiu Huang $^{b}$ \thanks{Corresponding author: hyqq@hunnu.edu.cn} } \\
\small $^{a}$ School of Mathematics, Hunan University, 
\small Changsha, 410082, China\\ 
\small $^{b}$College of Mathematics and Statistics, Hunan Normal University, 
\small Changsha, 410081, China 
}

%\footnote{Supported by the National Natural Science Foundation of China   (Grant No. 12271157, 12371346) and
%the Natural Science Foundation of Hunan Province, China (Grant No. 2022JJ30028, 2023JJ30178).
\date{}
\maketitle {\flushleft\large\bf Abstract}
A bipartite graph is  chordal bipartite if every cycle  of length at least six contains a chord. We determine the minimum size in 2-connected chordal bipartite graphs with  given order.



\section{Introduction}\label{se-1}
We consider finite simple graphs. For terminology and notation, we follow the book \cite{Bondy}. 
 Let $G$ be a graph. We denote by $V(G)$ and  $E(G)$  the vertex set and edge set of $G$. The \emph{order} of $G$ is  $|V(G)|$, and the \emph{size} is $|E(G)|$. We denote the set of all vertices adjacent to the vertex $v$ is $N_G(v)$. The \emph{degree} $d_G(v)$ of $v$ in $G$  is $|N_G(v)|$. A  graph is \emph{bipartite }if the vertex set can be partition into two independent sets. A bipartite  graph with  a bipartition $(X,Y)$ is called  a \emph{complete bipartite graph} if each vertex in $X$ is adjacent to each vertex in $Y$, and if $|X|=p$ and $Y=q$, we shall denote this complete bipartite graph by $K_{p,q}$.   A subset $S$ of $V(G)$ is called a \emph{vertex cut} if $G-S$ is disconnected, and a vertex-cut $S$ is called a \emph{$k$-vertex-cut} if $|S|=k$. A \emph{minimum vertex cut} of a graph is a vertex cut of smallest possible size. A vertex $v$ is called a \emph{cut vertex} if $G-v$ is disconnected. A \emph{component} of $G$ is a connected subgraph that is not part of any larger connected subgraph of $G$. Let $S$ be a vertex cut of $G$ and let $F$ be a component of $G-S$. The subgraph $H$ of $G$ induced by $S\cup V(X)$ is called an \emph{$S$-component}, denoted by $F^*$-component. For a non-complete graph $G$, its \emph{connectivity}, denoted by $\kappa(G)$, is defined to be the minimum value of $|S|$ over all vertex-cut $S$ of $G$, and $\kappa(G)=|V(G)|-1$ if $G$ is complete.  A graph $G$ is \emph{$k$-connected }if $\kappa(G)\ge k$. For a  integer $k\ge 1$,  we denote by $P_k$ a fixed path with $k$ vertices, and if $k \geq 3$ we denote by $C_k$ a fixed cycle with $k$ vertices.

%A vertex $u$ in $G$ is called \emph{simplicial}if either $d_G(u)=$ 0 or $N_G(u)$ forms a clique. 

Let $G$ be a graph. Let $S$ be a subset of $V(G)$. Denote by $G-S$ the graph obtained from $G$ by removing the vertices in $S$ together with all the edges incident with any vertex in $S$. We denote $G-v$ simply, where $S$ contains only the single vertex $v$. We define \emph{eliminating} an edge $u v$ from $G$ as $G-\{u, v\}$, denoted by $G \ominus u v$. Similarly, we define $G \ominus\left\{e_1, e_2, \ldots, e_s\right\}$ as the graph obtained from $G$ by sequentially eliminating $e_1, e_2, \ldots, e_s $.
 An edge $uv$ is called \emph{bisimplicial} if $N_G(u) \cup N_G(v)$ induces a complete bipartite graph with the bipartition $N_G(u), N_G(v)$ in $G$. Let $\sigma=[e_1,e_2,\ldots,e_k] $ be a sequence of pairwise nonadjacent edges of $H$.  We say that $\sigma$ is a \emph{perfect edge elimination order} for $G$ if each edge $e_i$ is bisimplicial in the $G \ominus \{e_1,\dots, e_{i-1}\}$ and $G \ominus \{e_1,\dots, e_{k}\}$ has no edge. A \emph{chord} $F$ of a cycle $C$ is an edge not in $E(C)$, with both of the end-vertices of $F$ lying on $C$.  A bipartite graph is called \emph{chordal bipartite} if every cycle of length at least six has a chord \footnote{In fact, chordal bipartite graphs are generally not chordal; therefore, although many publications refer to them by this misleading name, the term ``weakly chordal bipartite graph'' is a better alternative \cite[chapter 3 on page 41]{Brandstädt}.}  In {\cite{Golumbic}, two characterizations of chordal bipartite  graphs are mentioned: a bipartite graph $G$ is chordal bipartite if and only if every minimal edge separator induces a complete bipartite subgraph in $G$; moreover, a bipartite graph $G$ is chordal bipartite   if and only if every induced subgraph has a perfect edge elimination order. Various problems such as Hamiltonian cycle \cite{Müller} Steiner tree \cite{Müller2} and efficient domination \cite{Müller2} remain NP-complete on chordal bipartite graphs.

A well-known graph class closely related to bipartite chordal graphs is that of \emph{chordal graphs}, which are defined as graphs where every cycle of length greater than three has a chord. Recently, extremal problems related to chordal graphs have attracted the attention of graph theorists. Blair, Heggernes, Lima and Lokshtanov \cite{Blair} determined the maximum size of chordal graphs with bounded maximum degree and matching number.  Zhan and Zhang \cite{Zhan} determined the minimum size of a chordal graph with given order and
minimum degree. To study 1-planar graphs, we (the authors) along with Lv and Dong \cite{Zhang} determined the minimum size of a chordal graph with given order and connectivity. The following fundamental problem  arises naturally, analogous to the results on chordal graphs presented in \cite{Zhang}: \emph{What is the minimum size of a chordal bipartite graph with given order and connectivity?} The case of connectivity one in the problem   is trivial, as every connected graph with $n$ vertices and connectivity one has at least $n-1$ edges. Trees with order $n$, being chordal bipartite graphs, achieve the minimum size $n-1$. Therefore, we begin our study with 2-connected chordal bipartite graphs. The results of this paper will show that  the lower bound is no longer trivial for 2-connected chordal bipartite graphs with given order. 

To address the case of connectivity two,  we shall discover several useful tools. We first investigate how eliminating a bisimplicial edge affects the connectivity of a chordal bipartite graph, as described in Theorem \ref{thm1}. Prior to that, we mention some related research. For general graphs  with connectivity $k$, removing a vertex may reduce the connectivity, but it will be at least $k-1$. A  well-known result due to Chartrand, Kaugars, and Lick \cite{Chartrand} states that every $k$-connected graph $G$ with minimum degree $\delta(G) \geq\left\lfloor\frac{3}{2} k\right\rfloor$ has a vertex $u$ such that $G-u$ is still $k$-connected.  In 2003, Fujita and Kawarabayashi  \cite{Fujita}  considered a similar problem for eliminating an edge of a graph, and proved that every $k$-connected graph $G$ with minimum degree at least $\left\lfloor\frac{3}{2} k\right\rfloor+2$ has an edge $u v$ such that $G\ominus uv$ is still $k$-connected.  For chordal graphs $G$ with at least $k+2$ vertices and connectivity $k$, it is known that removing any simplicial vertex from $G$ does not decrease the connectivity \cite{Böhme}. Unlike chordal graphs, eliminating an bisimplicial edge from a chordal bipartite graph may reduce its connectivity, as formally stated in Lemma \ref{thm1}.

 For every 2-connected bipartite graph with $n$ vertices, the minimum size is $n$, and a cycle with $n$ vertices achieves the value $n$. Now for 2-connected chordal bipartite graphs, we provide a lower bound higher than $n$.

\begin{thm}\label{thm2}
Let $G$ be a 2-connected chordal bipartite graph of order $n\ge 4$ and size $m$.  Then 

\begin{equation*}
m \ge \begin{cases}
\frac{3}{2}n-\frac{3}{2}& \text{ if } n \text { is odd, }\\ \frac{3}{2}n-2 & \text{ if } n \text { is even. } \end{cases}
\end{equation*}

Moreover, the bounds are tight.
\end{thm}

\section{Proofs of main results}

The following lemma is obvious, based on the definition of connectivity.
\begin{lem}
  Let  $G$  be a graph with $\kappa(G)=k$. Let $S$ be a subset of $V(G)$ with size $s$. Then $\kappa(G-S)\ge k-s$. 
\end{lem}

Now, we introduce two lemmas on chordal bipartite graphs.

\begin{lem}[\cite{Golumbic}]\label{subch}
    Every induced subgraph of a chordal bipartite graph is  chordal bipartite.
\end{lem}

Analogous to the existence of a perfect vertex elimination order in chordal graphs, Golumbic and Goss \cite{Golumbic} proved that every chordal bipartite  graph with at least one edge has a bisimplicial edge, and subsequently established the following result.

\begin{lem}[\cite{Golumbic}]
  A bipartite graph $G$ is a chordal bipartite, then $G$  has a perfect edge elimination order.
\end{lem}

Note that, unlike chordal graphs, which can be characterized by a vertex elimination order \cite{Dirac}, the converse of the above lemma may not be true \cite{Golumbic}.

Before proving Theorem \ref{thm2}, we need to establish the following key lemma as a fundamental tool.

\begin{lem}\label{p1}
Let $G$ be a  chordal bipartite graph with  $\kappa(G)=k\ge 1$ and at least three vertices. Let $uv$ be a bisimplicial  edge of $G$. Assume that  $S$ is a vertex cut of $V(G\ominus uv)$. Let $F_1,F_2, \dots, F_t$ be the  component of $G\ominus  uv -S$ where $t\ge 2$. Denote $(N_G(u)\setminus v) $ and $(N_G(v)\setminus u)$ by $U$ and $V$, respectively.  Then the following statements  hold.
\begin{itemize}
  \item [(1)]$|U|\ge k-1 $ and $|V|\ge k-1$.
  \item [(2)] $G[U\cup V]$ is isomorphic to the complete bipartite graph with the bipartition $U,V$.
  %\item [(3)] If $|S|\le k-1$, then for each $S_i$, either $A\cap V(S_i) \ne \emptyset$  or  $B\cap V(S_i)\ne \emptyset $. 
  \item [(3)]  For two distinct components $F_i$ and $F_j$, we have  $U \cap V\left(F_i\right) =\emptyset$ or $V \cap V\left(F_j\right) = \emptyset$.
  \item [(4)] $|S|\ge k-1$.
   \item [(5)] If $|S|=k-1$ then $U = S$ or $V = S$.   Moreover, if $U =S$, then $V\cap S =\emptyset$ and  $V\cap F_i \ne \emptyset$ for each $i$ where $1 \le i \le t$; and symmetrically, if $V =S$, then $U\cap S =\emptyset$ and $U\cap F_i \ne \emptyset$ for each $i$ where $1 \le i \le t$.  
\end{itemize}
 
\end{lem}

%For (3), if not, then we conclude that $S$ is also a vertex cut of $G$. However, $|S|=k-2$, which contradicts the condition that $\kappa(G)\ge k$.

\begin{proof}
For (1), given that $\kappa(G)=k$ and that it is known $d_G(u) \geq \kappa(G)$, it follows that $d_G(u) \geq k$. Thus, $|U| \geq k-1$, and similarly, this holds for $V$ as well, as desired.

For (2), it follows that $uv$ is bisimplicial in $G$.

For (3), otherwise, by (2), $U \cap V(F_i)$ and $V \cap V(F_j)$ induce a complete bipartite graph, which contradicts that  $F_i$ and $F_j$ are distinct components.

 For (4), suppose that $|S|\le k-2$. Then by (1), we have  $|U|\ge k-1$ and  $|V|\ge k-1$, which implies that   $|U|>|S|$ and  $|V|>|S|$. Therefore, $U\setminus (U\cap S)\ne \emptyset $ and $V\setminus (V\cap S)\ne\emptyset $. Now  by (3) all vertices of  $U\setminus (U\cap S)$ and $U\setminus (U\cap S)$ lies within a unique component, namely $F_1$. Note that $t\ge 2$, which implies that $S$ is a vertex-cut of $G$ that separates the vertices of $F_2$ from the vertices of $G-(S\cup V(F_2)$, a contradiction to   the fact that $\kappa(G)=k$.  
 
For (5), suppose that $U\ne S$ and  $V\ne S$. Then by (1) we have $U\cap S\ne \emptyset$ and   $V\cap S\ne \emptyset$.  then  by (3) there exists  the  unique component, namely $F_1$, such that $U\setminus (U\cap S) \subseteq V(F_1)$ and $V\setminus (V\cap S) \subseteq V(F_1)$. Now $S$ is a $(k-1)$-vertex-cut of $G$, separating the vertices of $F_1$ from $G-(S\cup V(F_1)$,  a contradiction to   the fact that $\kappa(G)=k$.  So $U=S$ or  $V=S$. If $U=S$, then, since $G$ is bipartite, implying $U\cap V=\emptyset$, it follows that   $V\cap S =\emptyset$.  Now suppose that there exists a component $F_i$ such that  $V(F_i) \cap V=\emptyset$, then $S$ is a ($k-1$)-vertex-cut in $G$ that separates the vertices of $F_i$ from  the vertices of $G-(V(F_i)\cup S)$, a contradiction to the fact that $\kappa(G)=k$. Symmetrically, the same applies to $V=S$, and thus (5) holds.
\end{proof}

The structure described in Lemma \ref{p1}(5), where $V=S$, is shown in Fig. \ref{FIG1}.

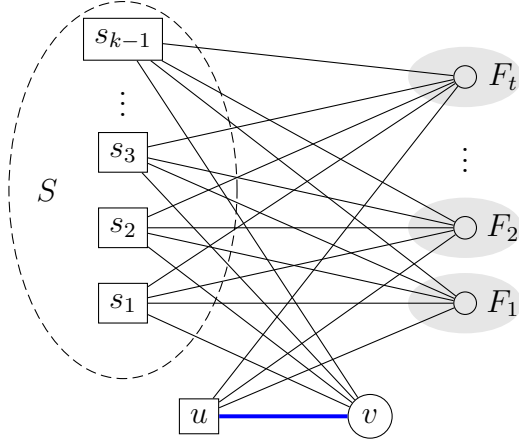
\begin{figure}
    \centering
    \begin{tikzpicture}[scale=0.5]
        % Define custom styles
        \tikzset{
            whitenode/.style={circle, draw=black, fill=white, minimum size=2pt, inner sep=3pt},  % White nodes as circles
            circlenode/.style={circle, draw=black, fill=gray!20, minimum size=2pt, inner sep=3pt}, % Custom circle style
            rectanglenode/.style={rectangle, draw=black, fill=white, minimum size=2pt, inner sep=4pt}  % Custom rectangle style
        }

        % Create nodes using the styles
         \draw[dashed, dash pattern=on 4pt off 2pt] (-4, 5) ellipse (3 and 5);
        \node[rectanglenode] (u) at (-2, -1) {$u$};    % White node with label "u"
        \node[whitenode] (v) at (2.5, -1) {$v$};  % Rectangle node with label "v"
        \node at (-6, 5) {$S$}; 
        \node [rectanglenode] (n2) at (-4, 2) {$s_1$};
		\node [rectanglenode] (n3) at (-4, 4) {$s_2$};
		\node [rectanglenode] (n4) at (-4, 6) {$s_3$};
		\node [rectanglenode] (n5) at (-4, 9) {$s_{k-1}$};

 \node at (5, 6) {$\vdots$};  % You can adjust the position as needed
 \node at (-4, 7.5) {$\vdots$}; 

\draw[fill=gray!20, draw=none] (5, 2) ellipse (1.5 and 0.8);  % Ellipse around n6
        \draw[fill=gray!20, draw=none] (5, 8) ellipse (1.5 and 0.8);   % Ellipse around n7
        \draw[fill=gray!20, draw=none] (5, 4) ellipse (1.5 and 0.8);   % Ellipse around n8 
        \node[circlenode] (n6) at (5, 2) {};  % No draw for invisible node
        \node[circlenode] (n7) at (5, 4) {};   % No draw for invisible node
        \node[circlenode] (n8) at (5, 8) {};   % No draw for invisible node
       \draw[line width=1.5pt,draw=blue] (u) -- (v);
        \node  (S1) at (6, 2) {$F_1$};
		\node  (S2) at (6, 4) {$F_2$};
		\node  (ST) at (6, 8) {$F_t$};

\foreach \n in {n2, n3, n4, n5} {
            \draw (v) -- (\n);
        }

 \foreach \n in {n6, n7, n8} {
            \draw (u) -- (\n);
        }
    \foreach \rect in {n2, n3, n4, n5} {
            \foreach \circ in {n6, n7, n8} {
                \draw (\rect) -- (\circ);
            }
        }        
    \end{tikzpicture}
    \caption{Illustration of the structure described in Lemma \ref{p1}(5) in the case of $V=S$, where $uv$ is an bisimplicial edge.}
    \label{FIG1}
\end{figure}

\begin{lem}\label{thm1}
Let $G$  be a  chordal bipartite graph with  $\kappa(G)=k$ where $k\ge 1$. Let $uv$ be a bisimplicial  edge of $G$. Then $\kappa(G \ominus uv)\ge k-1$.  
\end{lem}

\begin{proof}
If $k=1$, then $G\ominus u v$ is either connected or disconnected, and in this case, the lemma holds. Now we assume that $k\ge 2$. Suppose that  $S$ is a minimum vertex cut of $G\ominus  uv$ with $|S|\le k-2$ for $k\ge 2$. Now we note that $G$, $u v$ and $S$ satisfy all the conditions of Lemma \ref{p1}.  Thus by applying Lemma \ref{p1}(4), we have
 $|S|\ge k-1$,   a contradiction.  
\end{proof}

Now we prove Theorem \ref{thm2}.

\begin{proofoftheorem}[Proof of Theorem \ref{thm2}]
If $\kappa(G)\ge 3$, then it is known that $\delta(G)\ge \kappa(G)\ge 3$, which implies 
$m\ge \frac{3n}{2}$; thus the theorem clearly holds. Therefore, we assume below that $\kappa(G)=2$.  We  use induction on $n$.  In the base case, for $n=4$, since $\delta(G)\ge 2$, it is clear that $m\ge 4$, thus the theorem holds for this case. For $n=5$, similarly, we have $m\ge 5$. Since $\kappa(G)=2$,  for $m=5$, $G$ can only be $C_5$, which is not bipartite; therefore, $m \geq 6$, confirming that the theorem holds in this case. (In fact, it is easy to verify that for $n=4, G$ is isomorphic to $C_4$, and for $n=5$, $G$ is isomorphic to $K_{2,3}$). We assume that the theorem holds for every  chordal bipartite graph having connectivity two of order $n'$ such that $6\le n'<n$.
Let $e=uv$ be a bisimplicial edge of $G$.  By Theorem \ref{thm1}, $G \ominus  e$ is connected, and thus we can distinguish two cases.

\begin{case}
$\kappa(G\ominus e)\ge 2$.
\end{case}
For even $n$, we have $|V(G \ominus  e)|=n-2$, which remains even.  By Lemma \ref{subch}, $G \ominus e$ is chordal bipartite. Combining these with the inductive hypothesis, we have
\begin{align*}
m \geq & \ |E(G \ominus  e)| + 3 \\
    \geq & \  \frac{3}{2}(n-2) -2 + 3\\
    =&\frac{3}{2}n-2.
\end{align*}
For odd $n$, we have $|V(G \ominus  e)|=n-2$, which remains odd. Similarly, we have
\begin{align*}
m \geq & \ |E(G  \ominus  e)| + 3 \\
    \geq & \ \frac{3}{2}(n-2) - \frac{3}{2}+ 3\\
    =&\frac{3}{2}n-\frac{3}{2}.
\end{align*}

\begin{case}
$\kappa(G \ominus  e)= 1$.
\end{case}
 Let $s$ be cut vertex  of $G\ominus  e$.  Denote $(N_G(u)\setminus v) $ and $(N_G(v)\setminus u)$ by $U$ and $V$, respectively.  Let $F_1,F_2, \dots, F_t$ be the  component of $G\ominus  e -s$ where $t\ge 2$.   Since $\kappa(G)=2$, it follows that $d_G(u)\ge 2$ and $d_G(v)\ge 2$.  Furthermore, $G, u v$, and $s$ satisfy the conditions of Lemma \ref{p1} (at this point, replace $S$ with $\{s\}$ in the lemma). Thus, by Lemma \ref{p1}(5), without loss of generality,  we may assume  $V=\{s\}$. Consequently, $d_G(v)=2$. 

\begin{claim}\label{c1}
The vertex $u$ is connect to at least one vertex in each $F_i$ where $1\le i\le t$.  
\end{claim}
\begin{proof}
It follows directly from Lemma \ref{p1}(3).
\end{proof}

\begin{claim}\label{c2}
Every  $s$-component $F^*_i$ of $G\ominus  e$ is 2-connected.
\end{claim}
\begin{proof}

Suppose not. Then, $F^*_i$ contains a cut vertex $w\ne s$ ($s$ obviously cannot be a cut vertex of $F^*_i$ ). In this case, we shall assert that $w$ is a cut vertex of $G\ominus  e$. Let $T_1,T_2, \dots, T_l$ be  the components of $F^*_i$ where $l\ge 2$. We assume that the vertex cut (of $G\ominus  e$) $s \in V(T_1)$.
By Lemma \ref{p1}(2),  $s$ is adjacent to all vertices of $U$.  Combining this with the  assumption that $w$ is a vertex cut of $F^*_i$, we have $U\cap F^*_i  \subseteq T_1$. Now $w$ is also a vertex cut of $G$, separating the vertices of $T_2$ from the vertices of $G-(V(T_2)\cup \{w\})$, a contradiction to the fact that $\kappa(G)=2$.
\end{proof}

\begin{claim}
 $G- v$ is 2-connected.
\end{claim}
\begin{proof}
Suppose that there exists a cut vertex $w$ of  $G- v$. We now assert that $w\ne u,s$.   Recall that $G\ominus  e$ is connected. This means that $(G-v)-u$ is connected, so $u$ is a not vertex cut of  $G- v$.  By Claim \ref{c1}, $G- v-s$ has only one component, and thus $s$ is also not a vertex cut of  $G- v$. Thus,  $w$ must be within $V(G-v)\setminus \{u, s\}$. It must also lie in one of the components $F_i$ of $G\ominus e-s$. Now, $w$ would also be a cut vertex of the $w$-component $F^*_i$, which contradicts Claim \ref{c2}.
\end{proof}

Now for even $n$, and thus $|V(G -v)|$ is odd.  By Lemma \ref{subch}, $G -v$ is chordal bipartite. Recall that $d_G(v)=2$.  Combining these with the inductive hypothesis, we have
\begin{align*}
m = & \ |E(G -v)| + 2 \\
    \geq &  \frac{3}{2}(n-1) -\frac{3}{2} + 2\\
    =& \frac{3}{2}n-1\\
    \ge &\frac{3}{2}n-2.
\end{align*}
Similarly, for odd $n$, we have
\begin{align*}
m =& \ |E(G -v)| + 2 \\
    \geq &  \  \frac{3}{2}(n-1) - 2+ 2\\
    =& \frac{3}{2}n-\frac{3}{2}.
\end{align*}

Thus, the bounds of the theorem are proven. 

The following proves that these bounds are attainable. A  \emph{grid graph} $L(m, n)$ is the graph Cartesian product  two paths on $m$ and $n$ vertices. For even $n\ge 2$, the constructed graph $L(2, \frac{n}{2})$ is shown in Fig. \ref{f0}. For odd $n\ge 5$, we first select  $K_{2,3}$, and then, we attach a grid graph $L(2, \frac{n-3}{2})$ by identifying an edge of $K_{2,3}$, shown in Fig. \ref{f1}. It is easy to verify that the constructed graphs are chordal bipartite, and their sizes satisfy the lower bounds. Therefore, we have completed the proof.
\end{proofoftheorem}

\begin{figure}
\centering

\gridgraph{1}{4}
\begin{tikzpicture}[overlay]
        % Add ellipsis to the right of the top and bottom rows
        \node at (-5,2.2) {$\dots$};  % Dots at the top
        \node at (-5,-0) {$\dots$};  % Dots at the bottom
    \end{tikzpicture}

\caption{A chordal bipartite graph with even order $n$ and $\frac{3}{2}n-2$ edges. }
\label{f0}

\end{figure}
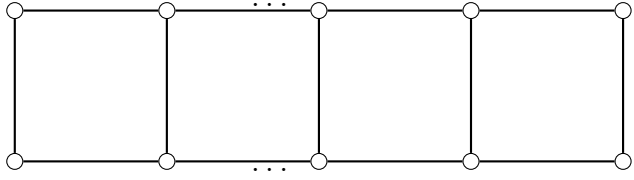

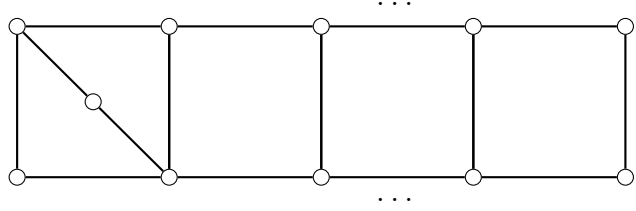
\begin{figure}
\centering
\begin{tikzpicture}
    % Define styles for nodes and edges
    \tikzset{
        node/.style={circle, draw=black, fill=white, minimum size=6pt, inner sep=0pt},
        edge/.style={draw=black, thick},
        label/.style={font=\small}
    }

    % Place the nodes using a loop
    \begin{pgfonlayer}{nodelayer}
        % Left square nodes + diagonal node
        \node[node] (A) at (-4, 2) {};  % Top-left
        \node[node] (B) at (-4, 0) {};  % Bottom-left
        \node[node] (E) at (-3, 1) {};  % Diagonal node in the left square

        % Define the other nodes in a loop (center and right squares)
        \foreach \i [count=\j from 0] in {-2, 0, 2, 4} {
            \node[node] (T\j) at (\i, 2) {};  % Top row
            \node[node] (B\j) at (\i, 0) {};  % Bottom row
        }

        % Dots for continuation
        \node (topdots) at (1, 2.3) {$\dots$};  % Dots at the top
        \node (bottomdots) at (1, -0.3) {$\dots$};  % Dots at the bottom
    \end{pgfonlayer}

    % Draw the edges using a loop
    \begin{pgfonlayer}{edgelayer}
        % Left square edges (manual due to diagonal)
        \draw[edge] (A) -- (B);  % Left vertical
        \draw[edge] (A) -- (T0);  % Top horizontal
        \draw[edge] (B) -- (B0);  % Bottom horizontal
        \draw[edge] (T0) -- (B0);  % Right vertical
        \draw[edge] (A) -- (E);  % Diagonal to middle node
       
        \draw[edge] (E) -- (B0);  % Middle to bottom center

        % Draw edges for the remaining squares (center and right squares)
        \foreach \i in {0, 1, 2} {
            \draw[edge] (T\i) -- (T\the\numexpr\i+1\relax);  % Top horizontal
            \draw[edge] (B\i) -- (B\the\numexpr\i+1\relax);  % Bottom horizontal
            \draw[edge] (T\i) -- (B\i);  % Left vertical
            \draw[edge] (T\the\numexpr\i+1\relax) -- (B\the\numexpr\i+1\relax);  % Right vertical
        }
    \end{pgfonlayer}
\end{tikzpicture}

\caption{A chordal bipartite graph with odd order $n$ and $\frac{3}{2}n-\frac{3}{2}$ edges. }
\label{f1}
\end{figure}

\section{Some remarks}
In this article, we determine the minimum size in 2-connected chordal bipartite graphs with  given order. However, the problem for higher connectivity remains unsolved. Based on some observations, we propose the following conjecture for the general case.
\begin{conj}
Let $G$ be a chordal bipartite graph with $n$ vertices and $m$ edges. Given a  integer  $k\ge 3$, if $\kappa(G)= k$, then there exists  an integer $N$ such that for $n\ge N$, it holds that $m\ge \frac{1+k}{2}n-O(1)$.
\end{conj}

When extending to higher values of connectivity $k$, we observe Case 2 (where $\kappa(G \ominus u v)=k-1$) in the inductive step  in Theorem \ref{thm2} that $G-v$ may no longer be $k$-connected for $k \geq 3$.
A example with connectivity $k=3$ shown in Fig. \ref{f3},  where  $\kappa(G\ominus uv)=2$ and $\kappa(G-v)=2$. 
Thus, a detailed investigation on the structure of $G-v$ will be a worthwhile endeavor for future research.

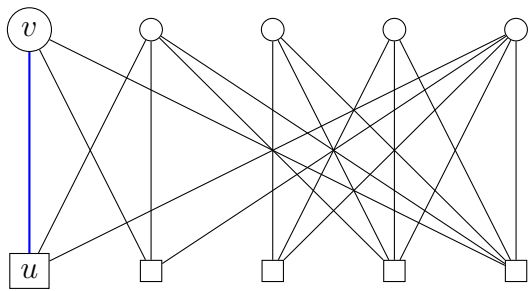
\begin{figure}
\centering
\begin{tikzpicture}[scale=0.8]
    % Define custom styles for nodes and edges
    \tikzset{
        whitenode/.style={circle, draw=black, fill=white, minimum size=6pt, inner sep=3pt},  % White circular nodes
        square/.style={rectangle, draw=black, fill=white, minimum size=6pt, inner sep=4pt},  % Square nodes
        blackedge/.style={draw=black},  % Standard black edges
        blueedge_thick/.style={draw=blue, thick},  % Blue thick edges
    }

    % Place the nodes
    \begin{pgfonlayer}{nodelayer}
        \node[whitenode] (0) at (-2, 2) {};  % Central node
        \node[whitenode] (1) at (-4, 2) {$v$};  % Left white node with label v
        \node[whitenode] (2) at (0, 2) {};  % White nodes on the top row
        \node[whitenode] (3) at (2, 2) {};
        \node[whitenode] (4) at (4, 2) {};

        \node[square] (5) at (-4, -2) {$u$};  % Left square node with label u
        \node[square] (6) at (-2, -2) {};  % Square nodes on the bottom row
        \node[square] (7) at (0, -2) {};
        \node[square] (8) at (2, -2) {};
        \node[square] (9) at (4, -2) {};
    \end{pgfonlayer}

    % Draw the edges between nodes
    \begin{pgfonlayer}{edgelayer}
        % Edges from node 0
        \draw[blackedge] (0) -- (5);
        \draw[blackedge] (0) -- (6);
        \draw[blackedge] (0) -- (8);
        \draw[blackedge] (0) -- (9);

        % Edges from node 1
        \draw[blackedge] (1) -- (6);
        \draw[blackedge] (1) -- (9);

        % Edges from node 2
        \draw[blackedge] (2) -- (7);
        \draw[blackedge] (2) -- (8);
        \draw[blackedge] (2) -- (9);

        % Edges from node 3
        \draw[blackedge] (3) -- (7);
        \draw[blackedge] (3) -- (8);
        \draw[blackedge] (3) -- (9);

        % Edges from node 4
        \draw[blackedge] (4) -- (5);
        \draw[blackedge] (4) -- (6);
        \draw[blackedge] (4) -- (7);
        \draw[blackedge] (4) -- (8);
        \draw[blackedge] (4) -- (9);

        % Blue thick edge from node 1 to node 5
        \draw[blueedge_thick] (1) -- (5);
    \end{pgfonlayer}
\end{tikzpicture}
\caption{A chordal bipartite graph with connectivity 3, where $uv$ is a bisimplicial edge.}
\label{f3}
\end{figure}

\section{Acknowledgment}
We wish to thank Professor Brendan McKay for  very helpful discussion with the first author. Notably, McKay has provided a new function for generating  chordal bipartite graphs in the software nauty (version 2.8.9) \cite{McKay}.
We claim that there is no conflict of interest in our paper. No data was used for the research
described in the article.

% Adjust the spacing in the bibliography


\begin{thebibliography}{99}
\small
\setlength{\itemsep}{-.8mm}  % Adjusts the space between bibliography items


\bibitem{Blair}
J.R.S. Blair, P. Heggernes, P.T. Lima and D. Lokshtanov, On the maximum number of edges in chordal graphs of bounded degree and matching number, \emph{Algorithmica} 84(2022), 3587-3602.
\bibitem{Böhme}
T. Böhme, J. Harant, M. Tkáč, More than one tough chordal planar graphs are hamiltonian,\emph{ J. Graph Theory}, 32(1999), 405-410.

\bibitem{Bondy}
J.A. Bondy, U.S.R. Murty, Graph Theory, Springer, Berlin, 2008.

\bibitem{Brandstädt}
A. Brandstädt, V.B. Le, J.P. Spinrad, Graph classes: a survey, in: SIAM Monographs on Discrete Mathematics and Applications, SIAM, Philadelphia, 1999. 
\bibitem{Chartrand} G. Chartrand, A. Kaigars, D.R. Lick, Critically $n$-connected graphs, \emph{Proc. Amer. Math. Soc.}
32(1972) 63-68.
\bibitem{Dirac}
G.A. Dirac, On rigid circuit graphs, \emph{Abh. Math. Sem. Univ. Hamburg} 25(1961), 71-76.

 \bibitem{Fujita} 
S. Fujita, K. Kawarabayashi, Connectivity keeping edges in graphs with large minimum
degree, \emph{J. Combin. Theory Ser. B} 98(2008) 805-811.
\bibitem{Golumbic}
  M.C. Golumbic, C.F. Goss,  Perfect elimination and chordal bipartite graphs, \emph{J.  Graph Theory} 2(1978), 155-163. 

\bibitem{McKay}
 B.D. McKay, A. Piperno, Practical graph isomorphism, II, \emph{J. Symbolic Comput.} 60(2014), 94–112.
\bibitem{Müller}
  H. Müller, Hamiltonian circuits in chordal bipartite graphs, \emph{Discrete Math.} 156(1996), 291-298. 

\bibitem{Müller2}
 H. Müller, A. Brandstädt,  The NP-completeness of Steiner tree and dominating set for chordal bipartite graphs, \emph{Theoret. Comput. Sci.} 53(1987), 257-265.

\bibitem{Zhan}  
  X.Z. Zhan, L.L. Zhang, The minimum size of a chordal graph with given order
and minimum degree, September  2024, arXiv:2409.10261.
\bibitem{Zhang}








L.C. Zhang, Y.Q. Huang, S.X. Lv and F.M. Dong, 4-connected 1-planar chordal graphs are Hamiltonian-connected, April 2024, arXiv:2404.156663.
\end{thebibliography}
\end{document}